\documentclass[11pt]{amsart}

\usepackage{amssymb,amsthm}
\newtheorem{thm}{Theorem}
\newtheorem{theorem}{Theorem}[section]

\newtheorem{lemma}[theorem]{Lemma}

\def\irr#1{{\rm  Irr}(#1)}

\def\ibr#1{{\rm  IBr}(#1)}

\def\ker#1{{\rm ker} (#1)}

\def\phi{\varphi}

\newcommand{\Bpi}[1]{{\rm B}_\pi (#1)}
\newcommand{\Ipi}[1]{{\rm I}_\pi (#1)}

\title[Kernels]{Kernels of Brauer characters and Isaacs' partial characters}

\begin{document}
	
\author[M. L. Lewis]{Mark L. Lewis}
\address{Department of Mathematical Sciences, Kent State University, Kent, OH 44242, USA}
\email{lewis@math.kent.edu}
	
\begin{abstract}  
In this paper, we prove a property of kernels of Brauer characters.  We propose a candidate for the kernels of Isaacs' partial characters, and we show that this candidate has the same property.		
\end{abstract}
	



\maketitle




\section{Introduction}

In this note, all groups are finite.  Lately, there has been a great deal of interest of determining what results that apply to the ordinary characters of a group also apply to the Brauer characters of a group.  So far, few of these results have involved kernels of characters translated to kernels of Brauer characters.  

In Section 2 of \cite{navbook}, the kernel of an irreducible Brauer character $\phi$ of $G$ is defined to be the kernel of ${\mathcal X}$ where ${\mathcal X}$ is an irreducible representation of $G$ over the field $F$ that affords $\phi$ where $F$ is a field of characteristic $p$.   With ordinary characters, the kernel of a character is also the kernel of a representation that affords the character.  In Definition 2.20 of \cite{text}, it is shown that the kernel of an ordinary character can be expressed in terms of the values of just the character.  In particular, one only needs the character table of the group to determine the kernel of an ordinary character.

For Brauer characters, we would also like to be able to express the kernel of a character in terms of the only the values of the character, and hence, only need the Brauer character table to determine the kernel.  We will see that we need the Brauer character table and all of the conjugacy classes of the group to determine the kernel of given Brauer character.  Given a group $G$, we will write $G^o$ for the set of $p$-regular elements of $G$.  

\begin{thm}\label{intro-one}
Let $G$ be a group and let $p$ be a prime.  Fix a Brauer character $\phi \in \ibr G$, set $L (\phi) = \langle g \in G^o \mid \phi (g) = \phi (1) \rangle$, and $K(\phi)/L(\phi) = {\bf O}_{p} (G/L(\phi))$.  Then $L (\phi)$ is a normal subgroup of $G$ and $K(\phi) = \ker \phi$.
\end{thm} 

In \cite{pisep}, Isaacs introduces a number of ideas which colloquially we currently call ``$\pi$-theory.''   The idea of $\pi$-theory is to extend the results that have been proven for Brauer characters of $p$-solvable groups to $\pi$-separable groups where $\pi$ is a set of primes.  To play the role of the Brauer characters, Isaacs defines $\pi$-partial characters.  

Given a $\pi$-separable group $G$, we take $G^o$ to be the set of $\pi$-elements of $G$, and Isaacs defines the $\pi$-partial characters to be the restrictions of the ordinary characters to $G^o$.  The irreducible $\pi$-partial characters are those $\pi$-partial characters that cannot be written as a sum of other $\pi$-partial characters, and we write $\Ipi G$ for the set of irreducible $\pi$-partial characters.  We note that most of the ideas of $\pi$-theory are also found in the recent monograph \cite{solv text}.  In addition there are several expository articles that discuss $\pi$-theory that are worth reading: \cite{arc}, \cite{pipart}, \cite{aust}, and \cite{nato}.

The key idea underlying Isaacs' $\pi$-theory is the Fong-Swan theorem.  The point is that the Fong-Swan theorem shows that the set ${\rm IBr} (G)$ for a $p$-solvable group $G$ can be determined entirely by $\irr G$.  Using the ideas from the Fong-Swan theorem, one can develop many of the results regarding the Brauer characters of $p$-solvable groups without referring to representations.  

When we replace the prime $p$ by a set of primes $\pi$, there is no natural way to replace a representation of characteristic $p$ with representation that depends on $\pi$.  Thus, Isaacs develops $\pi$-theory without referring to representations in characteristic $p$.  There has been an attempt to develop a representation theory for $\pi$ in \cite{pirep}, but we prefer to go with the representation-free approach.  The representation-free approach for developing $\pi$-partial characters is explored to its fullest in \cite{pipart}.

However, there are drawbacks to not having a representation.  In particular, we saw above that we define the kernel of a Brauer character $\phi$ to be the kernel of an $F$-representation that affords $\phi$ where $F$ is an algebraically closed field of characteristic $p$.  See page 39 of \cite{navbook} and it is shown there that this definition is well-defined.  However, when $\phi$ is a $\pi$-partial character, there is no representation in play, and the characters that restrict to $\phi$ need not have the same kernels.

There is one natural candidate for defining the kernel of $\phi$.  In \cite{pisep}, Isaacs finds a subset $\Bpi G$ of $\irr G$ that is in bijection with $\Ipi G$.  In particular, we write $G^o$ for the set of $\pi$-elements of $G$ and when $\chi \in \irr G$, we write $\chi^o$ for the restriction of $\chi$ to $G^o$; the map $\chi \mapsto \chi^o$ is a bijection from $\Bpi G$ to $\Ipi G$.  With this in mind, one logical way to define the kernel of $\phi$ would be $\ker \phi = \ker \chi$ where $\chi \in \Bpi G$ satisfies $\chi^o = \phi$.  We note however that over the years that Isaacs developed $\pi$-theory that the $B_\pi$-characters are not the only canonical set of lifts that he found, and a priori, it is not clear that all of the different ``canonical'' lifts give the same kernel.  See \cite{dpi}, \cite{nav}, \cite{nucl}, and \cite{cos} for other canonical lifts.

In \cite{pipart}, Isaacs shows that many (but not all) of the properties of the characters in $\Ipi G$ can be proved without reference to the characters in $\Bpi G$, and gives evidence that the partial character is the fundamental object here.  With this in mind, it makes sense to ask if the kernel of a character in $\Ipi G$ can be defined without referring to $\Bpi G$.  Furthermore, one can determine $\Ipi G$ just knowing the character table of $G$.  On the other hand, at this time, for many groups $G$ and choices of $\pi$, we do not know if the characters in $\Bpi G$ can be determined from the character table of $G$.  We will show that our candidate for the kernel of a character in $\Ipi G$ can be determined by the character table of $G$.  We will see that formula for finding the kernel of a $\pi$-partial character is similar to that of finding the kernel of a Brauer character.

Suppose $\phi$ is a $\pi$-partial character.  Define $L(\phi) = \langle x \in G^o \mid \phi (x) = \phi (1) \rangle$.  Then define $K (\phi)$ by $K(\phi)/L(\phi) = {\bf O}_{\pi'} (G/L(\phi))$.  The group $K(\phi)$ is our candidate for the kernel of $\phi$.  We note that one can read $L(\phi)$ from the values in the character table of $G$ and the classes in $G^o$, and then, using the character table and the classes outside of $G^o$ one can determine $K (\phi)$.   We will justify this definition is the correct one to be the kernel of $\phi$ by showing that when $\chi \in \Bpi{G}$ so that $\chi^o = \phi$ then $\ker {\chi} = K(\phi)$.

\begin{thm}\label{intro-two}
Let $\pi$ be a set of primes and let $G$ be a $\pi$-separable group.  If $\chi \in \Bpi G$ so that $\chi^o \in \Ipi G$, then $\ker \chi = K (\phi)$.
\end{thm}

Note that when $p$ is a prime and $G$ is a $p$-solvable group, the $\{ p \}'$-partial characters of $G$ coincide with the Brauer characters for $G$.  Theorems \ref{intro-one} and \ref{intro-two} show that we get the same kernel whether we view this function as a Brauer character or a $\{ p \}$'-partial character.

\section{Brauer Characters}

We first prove Theorem \ref{intro-one}.

\begin{proof}[Proof of Theorem \ref{intro-one}]
On page 39 of \cite{navbook}, Navarro defines the kernel of $\phi$ by $\ker \phi = \{ g \in G \mid \mathcal {X} (g) = I \}$ where ${\mathcal X}$ is an $F$-representation for $G$ that affords $\phi$ for a field $F$ of characteristic $p$.   Thus, we fix a $F$-representation ${\mathcal X}$ that affords $\phi$.  Let $X = \langle g \in G^o \mid \mathcal {X} (g) = I \rangle$.  Observe that $X \le L (\phi)$.  Suppose that $g \in G^o$ so that $\phi (g) = \phi (1)$.  Let $\phi (1) = n$ and let $\eta_1, \dots, \eta_n$ be the roots of unity so that $\phi (g) = \eta_1 + \cdots + \eta_n$.  Since $\eta_1, \dots, \eta_n$ are complex roots of unity, it is not difficult to see that $\eta_1 + \cdots + \eta_n = n$ if and only if $\eta_i = 1$ for all $i$ and this is true if and only if $g \in \ker {\mathcal {X}}$ (see Lemma 2.15 of \cite{text}).  It follows that $L (\phi) \le X$.  Hence, we have $L (\phi) = X$.   

Thus, we may assume that $X = L (\phi) = 1$.  This implies that $K (\phi) = {\bf O}_{p} (G)$.  In Lemma 2.32 of \cite{navbook}, it is proved that ${\bf O}_p (G) \le \ker {{\mathcal{X}}}$.  On the other hand, assuming $X(\phi) = 1$ implies that $\ker {{\mathcal{X}}}$ contains no nonidentity $p$-regular element.  In particular, $\ker {{\mathcal{X}}}$ will be a normal $p$-subgroup and so, $\ker {{\mathcal{X}}} \le {\bf O}_p (G)$.  We conclude that $\ker {{\mathcal{X}}} = {\bf O}_p (G) = K(\phi)$ and this proves the result.
\end{proof}

We next show that the intersection of all of the kernels of the irreducible Brauer characters of $G$ is ${\bf O}_p (G)$.  This result is not surprising, but we include it here to put it on the record.


\section{Isaacs $\pi$-partial characters}

We now turn to $\pi$-partial characters.  We begin by showing that if $\chi$ is a lift of a $\pi$-partial character $\phi$, then $L(\phi) \le \ker \chi$.

\begin{lemma} \label{one}
Suppose $\pi$ is a set of primes and $G$ is a $\pi$-separable group.  If $\phi$ is a $\pi$-partial character of $G$ and $\chi$ is an ordinary character of $G$ so that $\chi^o = \phi$, then $L(\phi) \le \ker \chi \le K(\phi)$.
\end{lemma}

\begin{proof}
Since $\chi^o = \phi$, we have that if $x \in G^o$ so that $\phi (x) = \phi (1)$, then $\chi (x) = \phi (x) = \phi (1) = \chi (1)$ and $x \in \ker {\chi}$.  This implies that the generators of $L (\phi)$ lie in $\ker {\chi}$, and so, $L (\phi) \le \ker {\chi}$.

Now, suppose that $\ker \chi$ is not contained in $K (\phi)$.  This implies that some prime $p$ in $\pi$ must divide $|\ker {\chi}:L(\phi)|$.  Hence, there must exist some element $x \in \ker {\chi} \setminus L(\phi)$ so that the order of $x$ is a power of $p$.  In particular, $x \in G^o$.  But now, since $x \in \ker \chi$, we have $\phi (x) = \chi (x) = \chi (1) = \phi (1)$, and this implies that $x \in L(\phi)$ which is a contradiction.  Thus, $\ker \chi \le K (\phi)$.
\end{proof}

Please note that it is not difficult to find examples of $\chi$ and $\phi$ as in Lemma \ref{one} where $L (\phi)$ and $K (\phi)$ are different.  We will find $\chi$ with $\ker \chi \ne K(\phi)$.  For example, take $C$ to be a cyclic $\pi$ group and let $H$ be a nontrivial $\pi'$-group. We take $G = C \times H$.  Let $\alpha \in \irr C$ be faithful, and observe that $\phi = (\alpha \times 1_H)^o$ is irreducible.  It is not difficult to see that $L (\phi) = 1$ and so $K (\phi) = H$.  Let $\theta$ be a nonprincipal linear character in $\irr H$, then if $\chi = \alpha \times \theta$, then $\chi^o = \phi$ and $\ker \chi = \ker \theta < H$.

When $\phi$ is irreducible, we will show that the character $\Bpi G$ associated to $\phi$ has the same kernel.  In other words, we prove Theorem \ref{intro-two}.


\begin{proof}[Proof of Theorem \ref{intro-two}]
By Lemma \ref{one}, we know that $L(\phi) \le \ker \chi \le K(\phi)$.  By \cite{preprint}, we know that $\chi \in \Bpi {G/L(\phi)}$.  Applying Corollary 5.3 of \cite{pisep}, we have ${\bf O}_{\pi'} (G/L(\phi)) \le \ker {\chi}$.  This implies that $K (\phi) \le \ker \chi$, and we conclude that $\ker \chi = K(\phi)$.
\end{proof}

We now refine the relationship between the kernel of a lift of $\phi$ and $K (\phi)$.

\begin{lemma}\label{three}
Let $\pi$ be a set of primes and let $G$ be a $\pi$-separable group.  If $\phi \in \Ipi G$, $\chi \in \irr G$ satisfies $\chi^o = \phi$, and $M = \ker \chi$, then $M \le K (\phi)$ and $K (\phi)/M = {\bf O}_{\pi'} (G/M)$.
\end{lemma}

\begin{proof}
By Lemma \ref{one}, we know that $L (\phi) \le M \le K (\phi)$.  This implies that $K(\phi)/M$ is a $\pi'$-group.  Since $K (\phi)/L (\phi) = {\bf O}_{\pi'} (G/L(\phi))$, it follows that ${\bf O}_{\pi'} (G/K(\phi)) = 1$.  This implies that $K(\phi)/M = {\bf O}_{\pi'} (G/M)$.
\end{proof}

We next show that $K (\phi)$ is uniquely determined by the lifts of $\phi$.

\begin{lemma}\label{four}
Let $\pi$ be a set of primes and let $G$ be a $\pi$-separable group.  If $\phi \in \Ipi G$, then $K (\phi)$ is the unique largest normal subgroup of $G$ such that $\phi (x) = \phi (1)$ for all $x \in K (\phi)^o$.
\end{lemma}

\begin{proof}
Let $K = K (\phi)$ and $L = L (\phi)$.  Observe that $K^o \le L$, so $K^o = L^o = G^o \cap L$.  We can find $\chi \in \irr G$ so that $\chi^o = \phi$.  By Lemma \ref{one}, we know that $L \le \ker {\chi}$.  Given $x \in K^o$, we have $x \in L$ so $\phi (x) = \chi (x) = \chi (1) = \phi (1)$.  It follows that $\phi (x) = \phi (1)$ for all $x \in K^o$.

Suppose now that $N$ is a normal subgroup so that $\phi (x) = \phi (1)$ for all $x \in N^o$.  By the definition of $L$, it follows that $N^o \le L$.  Notice that the subgroup generated by $N^o$ is ${\bf O}^{\pi'} (N)$, so ${\bf O}^{\pi'} (N) \le L$.  Thus, $NL/L \cong N/N \cap L$ is a quotient of $N/{\bf O}_{\pi'} (N)$, and thus, $NL/N$ is a $\pi'$-group.  This implies that $NL \le K$, and so, $N \le K$ as desired.
\end{proof}

We next obtain another characterization of $K(\phi)$.

\begin{lemma}\label{five}
Let $\pi$ be a set of primes and let $G$ be a $\pi$-separable group.  If $\phi \in \Ipi G$, then $K (\phi)$ is the unique largest normal subgroup of $G$ such that $\phi (x) = \phi (y)$ whenever $x$ and $y$ are $\pi$-elements of $G$ such that $K (\phi) x = K (\phi) y$.
\end{lemma}

\begin{proof}
Let $K = K(\phi)$.  Also, let $\chi \in \Bpi G$ so that $\chi^o = \phi$.  By Theorem \ref{intro-two}, we have that $\ker \chi = K$.  It follows that $\phi (x) = \chi (x) = \chi (Kx) = \chi (Ky) = \chi (y) = \phi (y)$.  Suppose now that $N$ is normal subgroup of $G$ so that $\phi (x) = \phi (y)$ whenever $x,y$ are $\pi$-elements in $G$ that satisfy $Nx = Ny$.  Notice that this implies that $\phi (x) = \phi (1)$ whenever $x \in N^o$, and by Lemma \ref{four}, this implies that $N \le K$.
\end{proof}

The Brauer character version of the next theorem is proved in Lemma 2.1 of \cite{super}.

\begin{theorem}
Let $\pi$ be a set of primes and let $G$ be a $\pi$-separable group.  Then $\bigcap_{\phi \in \Ipi G} \ker \phi = \bigcap_{\chi \in \Bpi G} \ker \chi = {\bf O}_{\pi'} (G)$.
\end{theorem}

\begin{proof}
Without loss of generality, we may assume ${\bf O}_{\pi'} (G) = 1$.   We work by induction on $|G|$.  If $G$ is simple, then all of the nonprincipal irreducible $\pi$-partial characters will be faithful, and the result will definitely hold.  Suppose that $G$ has at least two distinct minimal normal subgroups say $M_1$ and $M_2$.  Let $L_i/M_i = {\bf O}_{\pi'} (G/M_i)$ for each $i$.  We know that $M_1 \cap M_2 = 1$.  This implies that $L_1 \cap L_2$ has a trivial Hall $\pi$-subgroup, and so, $L_1 \cap L_2$ is a $\pi'$-group.  Since $L_1 \cap L_2$ is normal in $G$, we deduce that $L_1 \cap L_2 = 1$.  By induction, we have 
\[\displaystyle \bigcap_{\phi \in {\rm I}_{\pi} ({G/M_i})} \frac {\ker \phi}{M_i} = {\bf O}_p \left(\frac G{M_i}\right) = \frac {L_i}{M_i}\] 
for $i = 1, 2$.  Viewing the $\pi$-partial characters of the quotients as $\pi$-partial characters of $G$, we see that 
\[\displaystyle \bigcap_{\phi \in {\rm I}_{\pi} ({G/M_i})} \ker \phi \le L_i\] 
for $i = 1, 2$.  We obtain 
\[\displaystyle  \!\!\!\! \bigcap_{\phi \in {\rm I}_{\pi} (G)} \ker \phi \le   \!\!\!\! \bigcap_{\phi \in {\rm I}_{\pi} ({G/M_1})}  \!\!\!\! \ker \phi \  \bigcap  \!\!\!\!  \bigcap_{\phi \in {\rm I}_{\pi} ({G/M_2})}  \!\!\!\! \ker \phi  \le L_1 \cap L_2 = 1.\]
	
Thus, we may assume that $G$ has a unique minimal normal subgroup $M$.  Since ${\bf O}_{\pi'} (G) = 1$, we see that $|M|$ is divisible by some prime in $\pi$, and in particular, $M$ has a nonprincipal irreducible $\pi$-partial character $\mu$.  Let $\gamma$ be an irreducible constituent of $\mu^G$.  Since $\mu \ne 1_M$, we know that $M$ is not contained in $\ker \gamma$.  Because $M$ is the unique minimal normal subgroup of $G$, we see that $\ker \gamma = 1$, and thus, 
\[\displaystyle \bigcap_{\phi \in {\rm I}_{\pi} (G)} \ker \phi \le \ker \gamma = 1.\]  

Note that the bijection between the ${\rm I}_\pi (G)$ and ${\rm B}_\pi (G)$ preserves the kernels and so the intersections of those sets will be the same.  This proves the result.
\end{proof}

\end{document}